\documentclass[11pt,a4paper]{article}%
\usepackage[centertags]{amsmath}
\usepackage{amsfonts}
\usepackage{amssymb}
\usepackage{amsthm}
\usepackage{epsfig}
\usepackage{setspace}
\usepackage{ae}
\usepackage{eucal}
\usepackage[usenames]{color}%
\usepackage{graphicx}

\theoremstyle{plain}
\newtheorem{thm}{Theorem}[section]
\newtheorem{lem}[thm]{Lemma}
\newtheorem{cor}[thm]{Corollary}
\newtheorem{prop}[thm]{Proposition}
\theoremstyle{definition}

\theoremstyle{remark}

\newtheorem{rem}[thm]{Remark}

\newcommand{\E}{\mathbb{E}}

\begin{document}

\title{How to compute characteristic functions of affine processes via calculus of their operator symbols }
\author{ J\"org Kampen $^{1}$}
\maketitle

\begin{abstract}
The characteristic functions  of multivariate Feller processes with generator of affine type, and with smooth symbol functions have an explicit representation in terms of powers series with rational number coefficients and with monoms consisting of the symbol functions and formal derivatives of the symbol functions. The power series representations are convergent globally in time and on bounded domains of arbitratry size. Generalized symbol functions can be derived leading to power series expansions which are convergent on unbounded domains. The rational number coefficients can be efficiently computed by an integer recursion. As a numerical consequence characteristic functions of multivariate affine processes may be computed directly from the symbol function avoiding the generalized Riccati equation (an observation first made in \cite{BKS} in a more general context).  
\end{abstract}

\footnotetext[1]{Weierstrass Institute for Applied Analysis and Stochastics,
Mohrenstr. 39, 10117 Berlin, Germany, support by DFG Research Center Matheon is acknowledged.
\texttt{kampen@wias-berlin.de}.}

\section{Introduction}
Bochner told us to investigate Markov processes by the symbols of their associated infinitesimal generators. Recall the simple connection. From the probabilistic point of view we may start with a regular Feller process $$\left( (X^x_t)_{t\geq 0}, P^x\right)_{x\in {\mathbb R}^d}$$ with values in ${\mathbb R}^d$. Then the function
\begin{equation}\label{symbol}
\sigma (x,\xi):=-\lim_{t\downarrow 0}\frac{E^x\left(e^{i(X^x_t-x)\cdot \xi}\right)-1 }{t}
\end{equation}
is called the symbol of the process. The function
\begin{equation}
c_t(x,\xi)=E^x\left(e^{i(X^x_t-x)\cdot \xi}\right)
\end{equation}
entering the definition in (\ref{symbol}) is called the characteristic function. Now, if $(T_t)_{t\geq 0}$ is called the semigroup associated with $\left( (X^x_t)_{t\geq 0}, P^x\right)_{x\in {\mathbb R}^n}$, then we have
\begin{equation}
T_tf(x)=E^x\left(f(X^x_t)\right)=(2\pi)^{-n/2}\int_{{\mathbb R}^n}e^{ix\xi}c_t(x,\xi) \hat{f}(\xi)d\xi, 
\end{equation}
where $\hat{.}$ denotes the Fourier transform. It follows that the generator takes the form
\begin{equation}
Au(x)=\lim_{t\downarrow \infty}\frac{T_tu(x)-u(x)}{t}=-(2\pi)^{-n/2}\int_{{\mathbb R}^n}e^{ix\xi}\sigma(x,\xi)\hat{f}(\xi)d\xi
\end{equation}
Hence the symbol completely characterizes the process. However, recovering the characteristic function of a process  from its symbol is far from trivial. 
Especially it is far from trivial to say under which condition the symbol of a given partial integro-differential operator corresponds to a Markov process. H\"ormanders standard class of symbols of pseudo differential operators, namely functions $(x,\xi)\rightarrow \sigma(x,\xi)$ which satisfy
\begin{equation}
|\partial^{\alpha}_x\partial^{\beta}_{\xi}\sigma(x,\xi)|\leq c_{\alpha\beta}\sqrt{1+|\xi|^2}^{m-\rho|\alpha|+\delta|\beta|},
\end{equation}
where $0\leq \delta\leq \rho\leq 1$ do not fit very well in order to analyze this question. Symbolic calculus was generalized in \cite{BF}, and further generalized by \cite{H}, where the function $\sqrt{1+|\xi|^2}$ is replaced by more generalized weight functions. In this spirit weight functions were constructed such that the corresponding symbol class corresponds to a class of Feller processes (for an overview of more recent results cf. \cite{J} and references therein). Well it is clear, that symbolic calculus cannot fully characterize the class of Feller processes at least in its usual form because even in the class of Levy processes you find symbols which are not differentiable but correspond to fractional Laplacians for example. (Well, some symbolic calculus for non-differentiable symbol functions has been developed, but seems not suitable in the present case).  However, the class of Feller processes with differentiable symbols is quite considerable and for some fields of applications, especially finance, it seems to be satisfying. Well within this class the next question to be considered is how we can construct a certain class of Feller processes from their symbols, and the relation outlined above leads us to aim at the characteristic function of the process.
For processes $X$ where for each $t>0$ $X_t$ has an infinitely divisible distribution (such processes correspond Levy processes in law), an explicit formula for the characteristic function is known in form of a Levy-Kintchine formula, i.e.
\begin{equation}
\begin{array}{ll}
   \mathbb{E}\Big[e^{i\left\langle u,X^x_t-x\right\rangle } \Big]:=
e^{-\frac{1}{2}\left\langle Au,u\right\rangle t+i\left\langle r, u\right\rangle t+t\int_{{\mathbb R}^d}\left(e^{i\left\langle u,x\right\rangle}-1-i\left\langle u,x\right\rangle 1_D(x)\nu(dx) \right)}, 
\end{array}
\end{equation}
where $A$ is a symmetric nonnegative $d\times d$ matrix, $r\in{\mathbb R}^d$, and $\nu$ is a measure on ${\mathbb R}^d$ satisfying
\begin{equation}
\nu(\left\lbrace 0\right\rbrace )=0,~\mbox{and}~
\int_{{\mathbb R}^d\setminus \lbrace 0\rbrace} 
\min\left\lbrace |y|^2,1\right\rbrace \nu(dy)<\infty.
\end{equation}
For affine processes explicit formulas for the characteristic function are known only in special cases (cf. the example in \cite{DPS} for the case of affine processes without jumps). In the general case the computation of the (conditional) Fourier transform of an affine process is shown to be reducible to solving systems of general Riccati differential equations (cf. \cite{DFS}). Analytical solutions of the generalized Riccati differential equation systems are not known, and even numerically they are often difficult to solve, especially due to quadratic terms involved. In \cite{BKS} it is shown that for a class of affine equations a basis of analytic vectors can be found. In particular, this leads to recursive representations of characteristic functions in terms of constituents of the generator avoiding the generalized Riccati equations. In a special univariate case an explicit formula is derived.

In this paper we concentrate on the characteristic function and show that an explicit series in terms of the symbol and formal derivatives of the symbol with rational number coefficients can be derived. Convergence of the representation is shown.

\section{Explicit formulas for characteristic functions of affine processes}

For affine process $X^{0,x}$ with $X_{0}^{0,x}=x$ we search for an explicit representation of 
$\hat{p}(t,x,u)=\E[e^{\mathfrak{i}uX_{t}^{0,x}}]$. Here we denote $e^{iux}:=e^{\sum_{j=1}^d u_jx_j}$ for simplity of notation. We assume that the function $\widehat{p}$ is a global solution of the Cauchy problem
\begin{equation}\label{FBK1}
 \left\lbrace \begin{array}{ll}
 \frac{\partial\hat{p}}{\partial t}(t,x,u)= A\hat{p}\, (t,x,u),\\
\\
\hat{p}(0,x,u)=\exp(iux),\quad t\ge0,\quad x\in \Omega\subset\mathbb{R}^{d},
\end{array}\right.
\end{equation}
where $\Omega$ is some domain and 
\begin{align*}
&A[f](x)\equiv \frac{1}{2}\sum_{ij}a_{ij}(x)\frac{\partial^2}{\partial x_i\partial x_j}f(x)+\sum_i b_i(x)\frac{\partial}{\partial x_i}f(x) \\
&
 +\int_{\mathbb{R}^{d}\setminus\lbrace 0\rbrace}\left[  f(x+z)-f(x)-\frac{\partial
f}{\partial x}(x)\cdot z 1_{D}(z)\right]  \nu(x,dz).
\end{align*}
Here we assume that the equation is semielliptic, i.e.
$$(a_{ij})\geq 0,$$
and we assume that the operator is closed on some appropriate function space of its domain- in the case of a bounded $\Omega$ this is not a problem, but in general additional conditions on the measure $\nu$ may be needed (cf. \cite{BKS} for a discussion of that point). 
Furthermore $1_D$ is a indicator function of a bounded set $D$, $a_{ij}(x)$, $b_i(x)$ are affine functions, and $\nu(x,dz)$ is a Borel measure which is affine in $x$, i.e. we have the representations
\begin{equation}\label{coeff}
\begin{array}{ll}
a_{ij}(x)=a_{ij0}+\sum_{l=1}^d a_{ijl}x_l\\
\\
b_i(x)=b_{i0}+\sum_{l=1}^d b_{il}x_l\\
\\
\nu(x,dz)=1_{\Omega_0}\nu_{0}(dz)+\sum_{l=1}^dx_l1_{\Omega_l}\nu_l(dz),
\end{array}
\end{equation}
where $\Omega_0,\Omega_l\subset {\mathbb R}^d,~1\leq l\leq d$.
We define the symbol function
\begin{align*}
&\sigma:\Omega\times i{\mathbb R}^d\rightarrow {\mathbb C},\\
\\
&\sigma(x,iu):=\exp{(-iux)}A[\exp{iux}]\equiv -\frac{1}{2}\sum_{jk}a_{jk}(x)u_ju_k+\sum_j b_j(x)i u_j +\\
&
\exp{(-iux)}\int_{\mathbb{R}^{d}\setminus \lbrace 0\rbrace}\left[  \exp(iu(x+z))-\exp(iux)-iu\exp(iux)\cdot z 1_D(z)\right]  \nu(x,dz)=\\
\\
& -\frac{1}{2}\sum_{jk}a_{jk}(x)u_ju_k+\sum_j b_j(x)i u_j+\int_{\mathbb{R}^{d}\setminus \lbrace 0\rbrace}\left[  \exp(iuz)-1-iu\cdot z 1_D(z)\right]  \nu(x,dz)
\end{align*}
Since $\sigma$ is affine with respect to $x$ it is useful to define
\begin{equation}
\begin{array}{ll}
\sigma_l(iu):=&\partial_{x_l}\sigma(x,iu)=-\frac{1}{2}\sum_{jk}a_{jkl}u_ju_k+\sum_j b_{jl}i u_j\\
\\
&+\sum_{l=1}^d\int_{\mathbb{R}^{d}\setminus \lbrace 0\rbrace}\left[  \exp(iuz)-1-iu\cdot z 1_D(z)\right]\nu_l(dz)
\end{array}
\end{equation}
In the case $d=1$ we define
\begin{equation}
\begin{array}{ll}
\sigma_1(iu):=&\frac{\partial}{\partial \xi}\sigma(iu).
\end{array}
\end{equation}
\begin{rem}
In general we denote multivariate derivatives of order $\alpha=(\alpha_1,\cdots ,\alpha_d)$ of the symbol function with respect to the symbol variable $\xi$ by $\partial^{\alpha}_{\xi}\sigma$. Note that for $|\alpha|\geq 3$ multivariate derivatives of the symbol function are bounded, i.e. $|\partial^{\alpha}_{\xi}\sigma|<\infty$ if and only if ${\big |}\int_{{\mathbb R}^d\setminus \lbrace 0\rbrace}z^{\alpha}\nu(x,dz){\big |}< \infty$. 
\end{rem}

As a main result of the present paper we derive power series formulas for the characteristic functions of affine processes with rational number coefficient and (multivariate) monomials of the form
$$
\Pi_{|\epsilon_j|\leq k}\Pi_{l=1}^d\left(\partial^{\epsilon_j}_{\xi}\sigma (x,iu)\right)^{\alpha^k_j}\left(\partial_{\xi}^{\epsilon_j} \sigma_{l}(iu)\right)^{\beta^k_j},
$$
where the index $j$ denotes an enumeration of all multiindices $\epsilon_j$ of dimension $d$ and of order $k$ where $k$ varies over all nonnegative integer, i.e. $\epsilon_j=(\epsilon_{j1},\cdots ,\epsilon_{jd})$ and $\sum_{l=1}^d\epsilon_{jl}=k$, and $\alpha^k_j$ and $\beta_j$ denote nonnegative integers enumerated in accordance with the multiindices $\epsilon_j$.  
Obviously, we have
\begin{prop}
The symbol function
\begin{equation}
\begin{array}{ll}
\sigma:\Omega\times i{\mathbb R}^d \rightarrow {\mathbb C},~\Omega\subseteq {\mathbb R}^d~\mbox{open},\\
\\
\sigma(x,iu)=-\frac{1}{2}\sum_{jk}a_{jk}(x)u_ju_k+\sum_j b_j(x)i u_j\\
\\
\hspace{1.8cm}+\int_{\mathbb{R}^{d}}\left[  \exp(iuz)-1-iu\cdot z 1_D(z)\right]  \nu(x,dz)
\end{array}
\end{equation}
is globally bounded in the spatial variables $x$ if and only if either $\Omega$ is bounded or $b_{jl},a_{jkl}=0$ for $1\leq l\leq d$ and the measure $\nu=1_{\Omega_0}\nu_{0}(dz)+\sum_{l=1}^dx_l1_{\Omega_l}\nu_l(dz)$ has bounded $\Omega_l,~1\leq l\leq d$.
\end{prop}
First we shall derive formulas which are valid locally. Then in a second step we globalize the result to bounded spatial domain and global time (in the spirit of \cite{BKS}).
More precisely, in order to obtain a representation which is globally in time we make use the following transformation
\begin{equation}\label{trans}
\begin{array}{ll}
t:[0,1)\rightarrow [0,\infty),\\
\\
t(\tau):=\beta \ln\left(\tan(\frac{\pi}{4}+\frac{\pi}{4}\tau\right).
\end{array}
\end{equation}
Note that in the transformed time coordinates we get an equivalent Cauchy problem, namely 
\begin{equation}\label{FBK2}
 \left\lbrace \begin{array}{ll}
 \frac{\partial\hat{p}_*}{\partial \tau}(\tau,x,u)= \frac{\pi}{4}\frac{\beta}{sin\left(\frac{\pi}{2}+\frac{\pi}{2}\tau\right)}A\hat{p}_*\, (\tau,x,u),\\
\\
\hat{p}_*(0,x,u)=\exp(iux),\quad \tau\in [0,1)\quad x\in \Omega\subset\mathbb{R}^{d}.
\end{array}\right.
\end{equation}
\begin{rem}
In order to keep notation simple we shall abbreviate
\begin{equation}
\rho(\tau):=\frac{\pi}{4}\frac{\beta}{sin\left(\frac{\pi}{2}+\frac{\pi}{2}\tau\right)}
\end{equation}

\end{rem}

\begin{rem}
In \cite{BKS} The transformation $t(\tau)=-\beta \ln(1-\tau)$ is used for similar purposes. Our choice is guided by computational/numerical considerations which will be pointed out in a subsequent paper.
\end{rem}

 It makes sense to consider the univariate case first, because even that case is involved.
\subsection{The univariate case}

We introduce the following multiindex notation. For each positive integer $k$ we denote
\begin{equation}
\alpha^k=(\alpha^k_0,\cdots ,\alpha^k_{k-1}),~~\beta^k=(\beta^k_0,\cdots ,\beta^k_{k-1}).
\end{equation}
where the entries $\alpha^k_j,\beta^k_j$ are nonnegative integers.
For any positive integer $k$ the projection of a $k$-tuples $\alpha^k$ (resp. $\beta^k$ on a $l$-tuple for an nonnegative integer $l\leq k$ will be denoted by $\Pi^k_l$ with
\begin{equation}
\Pi^k_l \alpha^k=(\alpha^k_0,\cdots ,\alpha^k_{l-1}),~~(\mbox{resp.}~\Pi^k_l \beta^k=(\beta^k_0,\cdots ,\beta^k_{l-1})).
\end{equation}
Adding or subtracting integers to such multiindices is denoted as follows:
\begin{equation}
\alpha^k+i_j=(\alpha^k_0,\cdots ,\alpha^k_j+i_j,\cdots ,\alpha^k_{k-1}).
\end{equation}
Similar for $\beta^k$ and for subtraction, i.e. we have
\begin{equation}
\alpha^k-i_j=(\alpha^k_0,\cdots ,\alpha^k_j-i_j,\cdots ,\alpha^k_{k-1}).
\end{equation}
In our expansion only tuples with nonnegative entries will appear. However, we do not exclude negative integer entries from the beginning in order to keep the rules simple. Certain constants depending on the $k$-tuples $\alpha^k,\beta^k$ will just be defined zero if negative entries occur, and then they do not contribute to our expansion.
We have the following result
\begin{thm}
Locally, the following representation holds:
\begin{equation}\label{univarexp}
\begin{array}{ll}
\hat{p}(t,x,u)=\exp(iux)\times\\
\\
{\Big(}1+\sum_{k\geq 1,(\alpha^k,\beta^k)\in M^k}c_{(\alpha^k,\beta^k)}
\Pi_{j=0}^{k-1}\left( \partial^j_{\xi}\sigma(x,iu)\right)^{\alpha^k_j}\left(\partial^j_{\xi}\sigma_1(iu)\right)^{\beta^k_j}t^k{\Big)} 
\end{array}
\end{equation}
where
\begin{equation}
M_k=\left\lbrace (\alpha^k,\beta^k){\Bigg|}\sum_{j=0}^{k-1}\left(\alpha^k_j+\beta^k_j\right)=k~\&~\sum_{j=0}^{k-1}\beta^k_j\geq \sum_{j=1}^{k-1}\alpha^k_j \right\rbrace.
\end{equation}
Here, for $k\geq 2$ $c_{(\alpha_k,\beta_k)}=0$ if $(\alpha_k,\beta_k)\not\in M_k$, and
\begin{equation}\label{univaraff}
\begin{array}{ll}
c_{(1,0)}=1,~~c_{(0,1)}=0\\
\\
\mbox{ and for}~k\geq 2\\
\\
c_{(\alpha^{k+1},\beta^{k+1})}=\frac{1}{(k+1)!}~~\mbox{if}~~\alpha^{k+1}_0=k+1\\
\\
c_{\left( \alpha^{k+1},\beta^{k+1}\right) }=\frac{1}{(k+1)!}\times \\
\\
\sum_{j=0}^{k-1}\sum_{\sum_{i=0}^{k-1}\lambda^k_i=j}{\Pi_k \alpha^{k+1}-1_j+\lambda^k \choose \lambda_k} c_{\left( \Pi_k \alpha^{k+1}-1_j+\lambda^k,\Pi_k \beta^{k+1}-\lambda^k\right) }.
\end{array}
\end{equation}
Here, the natural restriction holds that $\lambda^k=(\lambda^k_0,\cdots ,\lambda^k_{k-1})$ are $k$-tuples such that
\begin{equation}
\Pi_k \alpha^{k+1}-1_j+\lambda^k \in M_k.
\end{equation}
(Alternatively we could define $c_{(\alpha^k,\beta^k)}$ with pairs of $k$-tuples $(\alpha^k,\beta^k)$ not in $M_k$ to be zero, of course). 
\end{thm}
\begin{rem}
The recursion is very easy to compute. Let us compute the approximation for the first few terms in full expansion(up to order 3 of (\ref{univaraff})). Let $d_k$ denote the sum of terms of order $k$ in time.  We have
\begin{equation}
\begin{array}{ll}
d_0=&1\\
\\
d_1=&\sigma(x,iu)\\
\\
d_2=&\frac{1}{2}\sigma(x,iu)^2+\frac{1}{2}(\partial_{\xi}\sigma(x,iu))\sigma_{1}(iu)\\
\\
d_3=&\frac{1}{3!}\sigma (x,iu)^3+\frac{3}{3!}\sigma(x,iu)(\partial_{\xi} \sigma (x,iu))\sigma_1(iu)\\
\\
&\frac{1}{3!}(\partial_{\xi}\sigma (x,iu))(\partial_{\xi}(\sigma_1(iu))\sigma_1(iu)+\frac{1}{3!}\left( \partial^2_{\xi}\sigma(x,iu)\right) \sigma_{1}(iu)^2.
\end{array}
\end{equation}
Well, higher order terms are easy to compute as well. The formular is coded by an affine triangle of rational numbers with $k$th row all numbers $c_{(\alpha^k,\beta^k)}$ where $\alpha^k,\beta^k$ are the $k$-tuples. We call that row the univariate affine triangle of the caracteristic function. The first row (terms of order $1$) contains just one number 
\begin{equation}
c_{(1,0)}=1
\end{equation}
(note that the tuple $(0,1)\not\in M_1$).
The second row contains two numbers
\begin{equation}
c_{((2,0),(0,0))}=\frac{1}{2},~~c_{((0,1),(1,0))}=\frac{1}{2}
\end{equation}
The third row contains four numbers
\begin{equation}
\begin{array}{ll}
c_{((3,0,0),(0,0,0))}=\frac{1}{6}~~c_{((1,1,0),(1,0,0))}=\frac{1}{2}\\
\\
c_{((0,1,0),(1,1,0))}=\frac{1}{6}~~c_{((0,0,2),(2,0,0))}=\frac{1}{6},
\end{array}
\end{equation}
and so on. The growth of the number of numbers in the rows of the affine triangle of the characteristic function is exponentially. However, with a simple computer program computation of order $20$ (about a million numbers) is no problem leading to formulas of unprecedented precision.
 
\end{rem}

\begin{rem}
 Note that in the second sum of \ref{univaraff} the cardinality of the second sum is
 \begin{equation}
  {j+k-1 \choose k-1},
 \end{equation}
a information which is useful for consideration of convergence.
\end{rem}

\begin{rem}
The representation is different from the special case considered in \cite{BKS}. In that case there is only one affine function in front of the integral operator, and we used this special structure to shift the affine functions within the symbol expansion. This is not possible in the general univariate affine case considered here. 
\end{rem}
Next we 'globalize' the result (quite in the spirit of \cite{BKS} by considering transformed variable). Note that in the case of transformed variables $t(\tau)=-\beta\ln(\tan(\frac{\pi}{4}+\frac{\pi}{4}\tau))$ we have symbol functions which depend on time, i.e.
\begin{align*}
&\sigma_{\rho}:[0,T_{\tau})\times \Omega\times i{\mathbb R}^d\rightarrow {\mathbb C},\\
\\
&\sigma_{\rho}(\tau ,x,iu):=\rho(\tau)\sigma(x,iu)
\end{align*}

 However, we let the symbol variable $\xi$ vary on the whole domain $i{\mathbb R}^d$ for the class of bounded symbol functions with bounded derivatives. As we observed, this is the case for processes with bounded moments and constant diffusion and drift coefficients.  
\begin{thm}
Let $\Omega\subseteq {\mathbb R}$ be some bounded domain of arbitraty size and let $T\in (0,\infty)$ some horizon. First assume that the symbol function $\sigma$ is bounded on $\Omega\times i{\mathbb R}^d$. Then there is a a parameter value $\beta >0$ such that on $[0,T_{\tau}]\times \Omega \times {\mathbb R}^d$ the following representation holds:
\begin{equation}\label{genunivarexp}
\begin{array}{ll}
\hat{p}_{\rho}(\tau,x,u)=\exp(iux){\Big(}1+\sum_{k\geq 1}d_k^*(\tau,x,iu)\tau^k{\Big)},
\end{array}
\end{equation}
where
\begin{equation}\label{genunivarexp1}
\begin{array}{ll}
d_k^*(\tau,x,iu):=\\
\\
\sum_{j=1}^k
\sum_{(m,l)\in M^{\tau}_{\beta kj}}\Pi_{r=1}^k\left(\rho^{m_r}\right)^{(l_r)}d_{\sum_{r=1}^j m_r}(x,iu)\frac{\sum_{r=1}^j m_r}{\sum_{r=1}^j(m_r+l_r)!},
\end{array}
\end{equation}
\begin{equation}\label{genunivarexp2}
M^{\tau}_{\beta kj}:=\left\lbrace (m,l):=(m_1,l_1,\cdots ,m_j,l_j)\in {\mathbb N}_0^{2j}|\sum_{r=1}^j(m_r+l_r)=k\right\rbrace ,
\end{equation}
and for $l=\sum_{r=1}^j m_r$
\begin{equation}\label{genunivarexp3}
d_l(x,iu)=\sum_{(\alpha^k,\beta^k)\in M^k}c_{(\alpha^k,\beta^k)}
\Pi_{j=0}^{k-1}\left( \partial^j_{\xi}\sigma(x,iu)\right)^{\alpha^k_j}\left(\partial^j_{\xi}\sigma_1(iu)\right)^{\beta^k_j}
\end{equation}
with the same recursion for $c_{(\alpha^k,\beta^k)}$ as before. Furthermore $\left( \rho^m\right)^{(l)}$ denotes the $l$th derivative of the $m$th power of the function $\rho$. Explicit formulas for that are easily available by elementary calculation. 

 Note that this implies a representation of the characteristic function which is global in time by replacing $\tau$ by $\tau(t)$ in the representation above.
\end{thm}

In the general case we have
\begin{cor}
Let $\Omega,~ S\subset {\mathbb R}$ be bounded domains of arbitraty size and let $T\in (0,\infty)$ some horizon.  Then there is a a parameter value $\beta >0$ such that on $[0,T_{\tau}]\times \Omega \times S$ the representation described in in (\ref{genunivarexp}), (\ref{genunivarexp1}), (\ref{genunivarexp2}), (\ref{genunivarexp3}) holds.
\end{cor}
The latter corollary shows that univariate affine processes with smooth symbols can be approximated on finite domains with accuracy of any practical interest by the power series described in theorem 2.9.. However, in Section 4 we shall define generalized symbol functions and describe generalizations of theorem 2.9 where we describe power series in terms of generlized symbol functions and their dervatives which converge on the whole domain (i.e. $\Omega \subseteq {\mathbb R}^n$ may be unbounded or equal to ${\mathbb R}^d$ and $S=i{\mathbb R}^d$ in corollary 2.10).  

\subsection{The multivariate case}
The difference between the univariate and the multivariate case is not that big as one may think. Well, we shall minimize this difference using a certain trick (which is also useful when programming the formulas of this article). For each positive integer $k$ we enumerate the $d$-dimensional multiindices in a list
\begin{equation}
\left\lbrace \epsilon_j=(\epsilon_{j1},\cdots ,\epsilon_{jd})|\sum_{r=1}^d \epsilon_{jr}=k~\&~j\in\left\lbrace 0,\cdots N_k-1\right\rbrace \right\rbrace 
\end{equation}
where $N_k$ is a natural number which euals the cardinality of the $d$-tuples of order $k$. In addition we need some '$k$-tuples of $d$-tuples' 
 For each positive integer $k$ we denote
\begin{equation}
\lambda_d^k=(\lambda^k_{d0},\cdots ,\lambda^k_{dk-1}),~~\beta^k_d=(\beta^k_{d0},\cdots ,\beta^k_{d(k-1)}).
\end{equation}
where the entries $\lambda^k_{dj},\beta^k_{dj}$ are $d$-tuples, i.e.
\begin{equation}
\lambda^k_{dj}=\left(\lambda^k_{dj1},\cdots, \lambda^k_{djd} \right),~\beta^k_{dj}=\left(\beta^k_{dj1},\cdots, \beta^k_{djd} \right),
\end{equation}
and with nonnegative numbers $\lambda^k_{djl},\beta^k_{dj1}~1\leq l\leq d$. 
Adding or subtracting integers to such multiindices is defined as before. Finally we have to adjust the tuples and the projection operators for tuples a bit. Instead of $\alpha^k$ in the univariate case we now have tuple
$$
\alpha^{N_k}=(\alpha_0,\cdots,\alpha_{N_k-1}).
$$
For all $k\in {\mathbb N}$ we define the projection operator $\Pi_{N_k}$ on the $N_{k+1}$ tuples of nonnegative integers which is determined by the relation 
$$
\Pi_{N_k}(\alpha_0,\cdots ,\alpha_{N_{k+1}})=(\alpha_0,\cdots ,\alpha_{N_{k}}).
$$
Furthermore for each $k\in {\mathbb N}$ we define
$$
\Pi^d_k \beta_d^{k+1}=\Pi^d_k(\beta^k_{d0},\cdots ,\beta^k_{dk})=(\beta^k_{d0},\cdots ,\beta^k_{d(k-1)})
$$
We have the following result
\begin{thm}
Locally, the following representation holds:
\begin{equation}
\begin{array}{ll}
\hat{p}(t,x,u)=&\exp(iux){\Big(}1+\sum_{k\geq 1,(\alpha^{N_k},\beta_d^k)\in M^k_d}c_{(\alpha^{N_k},\beta_d^k)}\times\\
\\
&\Pi_{j=0}^{N_k-1}\Pi_{l=1}^d\left(\partial_{\xi}^{\epsilon_j} \sigma(x,iu)\right)^{\alpha^{N_k}_j}
\left(\partial_{\xi}^{\epsilon_j} \sigma_l(iu)\right)^{\beta^k_{djl}}t^k{\Big)},
\end{array}
\end{equation}
where
\begin{equation}
M^k_d=\left\lbrace (\alpha^k,\beta^k_d){\Bigg|}\sum_{j=0}^{N_k-1}\sum_{l=1}^d\left(\alpha^k_j+\beta^k_{djl}\right)=k~\&~\sum_{j=0}^{N_k-1}\sum_{l=1}^d\beta^k_{djl}\geq \sum_{j=1}^{k-1}\alpha^k_j \right\rbrace.
\end{equation}
Here, for $k\geq 2$ $c_{(\alpha_k,\beta_k)}=0$ if $(\alpha_k,\beta_k)\not\in M^k_d$, and
\begin{equation}\label{multivaraff}
\begin{array}{ll}
c_{(1,(0,\cdots ,0))}=1 (\mbox{all other tuples of order $1$ equal zero}\\
\\
\mbox{ and for}~k\geq 2\\
\\
c_{(\alpha^{N_{k+1}},\beta_d^{k+1})}=\frac{1}{(k+1)!}~~\mbox{if}~~\alpha^{N_{k+1}}_0=k+1\\
\\
c_{\left( \alpha^{N_{k+1}},\beta_d^{k+1}\right) }=\frac{1}{(k+1)!}\times\\
\\
\sum_{j=0}^{N_k-1}\sum_{\sum_{r=0}^{N_k-1}\lambda^k_{dlr}=\epsilon^k_{jl},l=1,\cdots d}\Pi_{l=1}^d{\Pi_{N_k} \alpha^{N_{k+1}}-1_j+\lambda^k_{dl} \choose \lambda^k_{dl}}\times \\
\\
\hspace{6cm} c_{\left( \Pi_{N_k} \alpha^{N_{k+1}}-1_j+\lambda^k_{dl},\Pi^d_k \beta_d^{k+1}-\lambda^k_{dl}\right) }.
\end{array}
\end{equation}

\end{thm}

\begin{proof}
Polynomial growth of the derivatives witb respect to the symbol variable means that the symbol function is analytic (this can be checked easily by estmating the remainder term of the multivariate Taylor formula). Hence we have
\begin{equation}
\begin{array}{ll}
\sigma(x,iu)=-\sum_{i,j=1}^da_{jk}u_ju_k+\sum_{j=1}^db_j(x)iu
+\sum_{|\delta|\geq 1}
\frac{m_{\delta}(x)}{\delta !}(iu)^{\delta}\\
\\
=:\sum_{|\delta|\geq 1}\frac{\gamma_{\delta}(x)}{\delta !}(iu)^{\delta},
\end{array}
\end{equation}
where $\delta$ denotes $d$-tuples with nonnegative intgegers, and for $|\delta|\geq 2$
\begin{equation}
m_{\delta}(x):=\int_{{\mathbb R}^d\setminus\{0\}}z^{\delta}\nu(x,dz).
\end{equation}
Plugging in the ansatz $\hat{p}(t,x,u)=\exp(iu^Tx)\left(1+d_k(u,x)t^k\right)$ leads to the
recursive formula
\begin{equation}
\begin{array}{ll}
d_{k+1}=\frac{1}{k+1}\exp(-iux)\sum_{\delta\geq 0}\frac{\gamma_{\delta}}{\delta!}\partial^{\delta}\left(\exp(iux) d_k\right)
\end{array}
\end{equation}
Using multivariate Leibniz derivative rule we get
\begin{equation}
\begin{array}{ll}
d_{k+1}&=\frac{1}{k+1}\exp(-iux)\sum_{\delta\geq 0}\frac{\gamma (x)_{\delta}}{\delta!}\sum_{\epsilon\leq \delta}{\delta \choose \epsilon} (iu)^{\delta-\epsilon}\exp(iux)\partial^{\epsilon}\left(  d_k\right)\\
\\
&=\sum_{|\epsilon|\leq k}\partial^{\epsilon}_{\xi}\sigma (x,iu)\frac{1}{\epsilon !}\partial^{\epsilon}_{x}\left(  d_k\right) 
\end{array}
\end{equation}

\end{proof}
\begin{rem}
The recursion is very easy to compute. Let us compute the approximation for the first few terms in full expansion(up to order 3 of (\ref{univaraff})). Let $d_k$ denote the sum of terms of order $k$ in time.  We have
\begin{equation}
\begin{array}{ll}
d_0=&1\\
\\
d_1=&\sigma(x,iu)\\
\\
d_2=&\frac{1}{2}\sigma(x,iu)^2+\frac{1}{2}\sum_{l=1}^d(\partial_{\xi_l}\sigma(x,iu))\sigma_{l}(iu)\\
\\
d_3=&\frac{1}{3!}\sigma (x,iu)^3+\frac{3}{3!}\sum_{l=1}^d\sigma(x,iu)(\partial_{\xi_l} \sigma (x,iu))\sigma_l(iu)\\
\\
&+\frac{1}{3!}\sum_{k,l=1}^d(\partial_{\xi_k}\sigma (x,iu))(\partial_{\xi_l}(\sigma_1(iu))\sigma_1(iu)\\
\\
&+\frac{1}{3!}\sum_{k,l=1}^d\left( \partial^2_{\xi_k\xi_l}\sigma(x,iu)\right) \sigma_{1}(iu)^2,
\end{array}
\end{equation}
and so on.
\end{rem}
Again this theorem may be globalized leading to

\begin{thm}
Let $\Omega\subseteq {\mathbb R}^d$ be some bounded domain of arbitraty size and let $T\in (0,\infty)$ some horizon. First assume that the symbol function $\sigma$ is bounded on $\Omega\times i{\mathbb R}^d$. Then there is a a parameter value $\beta >0$ such that on $[0,T_{\tau}]\times \Omega \times {\mathbb R}^n$ the following representation holds:
\begin{equation}\label{genmultivarexp}
\begin{array}{ll}
\hat{p}_{\rho}(\tau,x,u)=\exp(iux){\Big(}1+\sum_{k\geq 1}d_k^*(\tau,x,iu)\tau^k{\Big)},
\end{array}
\end{equation}
where
\begin{equation}\label{genmultivarexp1}
\begin{array}{ll}
d_k^*(\tau,x,iu):=\\
\\
\sum_{j=1}^k
\sum_{(m,l)\in M^{\tau}_{\beta kj}}\Pi_{r=1}^k\left(\rho^{m_r}\right)^{(l_r)}d_{\sum_{r=1}^j m_r}(x,iu)\frac{\sum_{r=1}^j m_r}{\sum_{r=1}^j(m_r+l_r)!},
\end{array}
\end{equation}
\begin{equation}\label{genmultivarexp2}
M^{\tau}_{\beta kj}:=\left\lbrace (m,l):=(m_1,l_1,\cdots ,m_j,l_j)\in {\mathbb N}_0^{2j}|\sum_{r=1}^j(m_r+l_r)=k\right\rbrace ,
\end{equation}
and for $s=\sum_{r=1}^j m_r$
\begin{equation}\label{genmultivarexp3}
\begin{array}{ll}
d_s(x,iu)=&\sum_{(\alpha^{N_s},\beta_d^s)\in M^s_d}c_{(\alpha^{N_s},\beta_d^s)}\times\\
\\
&\Pi_{j=0}^{N_s-1}\Pi_{l=1}^d\left(\partial_{\xi}^{\epsilon_j} \sigma(x,iu)\right)^{\alpha^{N_s}_j}
\left(\partial_{\xi}^{\epsilon_j} \sigma_l(iu)\right)^{\beta^s_{djl}}
\end{array}
\end{equation}
with the same recursion for $c_{(\alpha^k,\beta^k)}$ as before. Furthermore $\left( \rho^m\right)^{(l)}$ denotes the $l$th derivative of the $m$th power of the function $\rho$. Explicit formulas for that are easily available by elementary calculation. 

 Note that this implies a representation of the characteristic function which is global in time by replacing $\tau$ by $\tau(t)$ in the representation above.
\end{thm}
In the general case we have
\begin{cor}
Let $\Omega,~ S\subset {\mathbb R}$ be bounded domains of arbitraty size and let $T\in (0,\infty)$ some horizon.  Then there is a a parameter value $\beta >0$ such that on $[0,T_{\tau}]\times \Omega \times S$ the representation described in in (\ref{genunivarexp}), (\ref{genmultivarexp1}), (\ref{genmultivarexp2}), (\ref{genmultivarexp3}) holds.
\end{cor}
The latter corollary shows that univariate affine processes with smooth symbols can be approximated on finite domains with accuracy of any practical interest by the power series described in theorem 2.9.. However, in Section 4 we shall define generalized symbol functions and describe generalizations of theorem 2.9 where we describe power series in terms of generlized symbol functions and their dervatives which converge on the whole domain (i.e. $\Omega \subseteq {\mathbb R}^n$ may be unbounded or equal to ${\mathbb R}^d$ and $S=i{\mathbb R}^d$ in corollary 2.10).  
\section{Convergence by counting terms}

The structure of the explicit formulas makes it possible to prove convergence by counting the number of terms in the explicit expansion. We first have a look at the univariate case and then will generalize to the multivariate case.

Let $\Pi^n_k$ with $1\leq k\leq n,~k,n\in {\mathbb N}$ denote the number of terms of order $n$ in time, i.e. coefficients of $t^n$ in the expansion (\ref{univarexp}), and of order $k$ in the spatial variables, i.e.  the monoms of the power series (\ref{univarexp})) of order $k$, where $k=k_1+k_2+k_3$, and where $k_1$ is the exponent of drift coefficient functions or derivatives of drift coefficients involved, and $k_2$ is the exponent of diffusion coefficient functions or derivatives of diffusion coefficient functions involved, and $k_3$ is the exponent of integral terms with an affine jump measure or derivatives of integral terms with an affine jump measure involved. Note that in difference to the expansion (\ref{univarexp}) we collect {\it all} terms of order $k$ in spatial variables for some time order $n$, i.e. we make no difference between monoms where mutually different derivatives of the symbol function occurr.

Then we may determine that numbers in an integer recursion which we call the 'Affine Counting triangle'. In the following triangle the number of the row denotes the $n$, and the number of the column denotes the $k$ of $\Pi^n_k$. The first rows of the Affine Counting triangle are
\begin{equation}
\begin{array}{ll}
&1\\
& 1~~1\\
& 1~~3~~~2\\
& 1~~6~~~10~~3\\
& 1~~10~~34~~45~~4\\
&\cdots
\end{array}
\end{equation}

Elementary consideration lead us to
\begin{lem}
The number of terms of order $n$ in time and of order $n-k$ w.r.t. the spatial variables in the explicit formula is given by the recursion
\begin{equation}
\begin{array}{ll}
\Pi^n_n=1;~~n\in {\mathbb N}\\
\\
\Pi^n_{n-k}=\sum_{l=0}^k{n-1-l \choose k-l}\Pi^{n-1}_{k-1-l}~~k,n\in {\mathbb N},~1\leq k \leq n
\end{array}
\end{equation}

\end{lem}

\begin{lem}
An elementary induction argument leads us to
\begin{equation}
\sum_{k=0}^{n-1}\Pi^n_{n-k}\leq n!
\end{equation}

\end{lem}

\begin{rem}
Indeed, the latter result is a rough estimate. Defining
\begin{equation}
R_k=\sum_{k=0}^{n-1}\Pi^n_{n-k}\leq n!,
\end{equation}
from the first five rows of the Affine Counting triangle we see that
\begin{equation}
\begin{array}{ll}
\frac{R_1}{1!}=1,~\frac{R_2}{2!}=\frac{2!}{2!}=1,~\frac{R_3}{3!}=\frac{6}{3!}=1,\\
\\
~\frac{R_4}{4!}=\frac{20}{4!}=\frac{5}{6},~
\frac{R_5}{5!}=\frac{94}{5!}=\frac{47}{60},~~\cdots 
\end{array}
\end{equation}
It can be shown that $\frac{R_k}{k!}\downarrow 0$ as $k\uparrow \infty$, but I have not determined the exact ratio of convergence yet. Such an analysis may be important for numerical and computational purposes. However the rough estimate by the faculty is enough for our purposes.
\end{rem}

We have
\begin{lem}
The power series (\ref{univarexp}) converges locally to the characteristic function of the process $X^{\mbox{{\tiny aff}}}_1$.
\end{lem}

\begin{proof}
First we consider the case where the process has a bounded symbol function. Then we have less then $k!$ monoms of order $k$ each one has the factor $\frac{1}{k!}$ and consists of a finite product of the symbol function and drivatives of the symbol function. Moreover we have seen that if the multiplicand $\left( \partial_{\xi}^j\sigma(x,iu)\right)^{\alpha^k_j}$ with $0\leq j\leq k-1$ in such a symbol is accompanied by a multipicand which is a powere of $\sigma_{l}(iu)$ with same exponent. Since $\partial_{\xi}^j\sigma(x,iu)\leq c^j \sigma (x,iu)$ each momom is bounded by $c^k$ for some $0<c<1$. This means $|d_k|\leq c^k$ and for $\tau\in [0,1)$ the whole series is absolutely bounded by $\frac{1}{1-c}$.
\end{proof}

\section{Generalized symbol function}
In general it is known that the characteristic function of regular affine processes $\left(X,(P_{x})_{x\in \Omega} \right) $ with associated semigroup $(T_t)_{t\in {\mathbb R}_+}$ have a representation of the form
\begin{equation}\label{semi}
T_t\exp(iux)=\exp\left( \phi(t,u)+\psi(t,u)x)\right) 
\end{equation}
where the functions $\phi$ and $\psi$ are determined by generalized Riccati equations (cf. \cite{DFS} for details). Well these are difficult to solve (even numerically), but in some special cases it can be done and was done. In such cases we may use the explicit formula and define generalized symbol functions from it. The the procedure is as follows.
Assume that for some Cauchy problem with reduced symbol function you can solve the Cauchy problem 
For example you may think of some equations of type
\begin{equation}\label{DiffCauchy}
 \left\lbrace \begin{array}{ll}
 \frac{\partial\hat{p}}{\partial t}(t,x,u)= \left( \frac{1}{2}\sum_{ij}a_{ij}(x)\frac{\partial^2}{\partial x_i\partial x_j}+\sum_i b_i(x)\frac{\partial}{\partial x_i}\right) \hat{p}\, (t,x,u)=0,\\
\\
\hat{p}(0,x,u)=\exp(iux),\quad t\ge0,\quad x\in \Omega\subset\mathbb{R}^{d},
\end{array}\right.
\end{equation}
where $\Omega$ is some domain and $a_{ij}$ and $b_i$ are affine function as in (\ref{coeff}). Examples for this case are given below. However there are also examples of explicit solutions in form of (\ref{semi}) if there are jump terms with Levy measure (cf. \cite{DFP}). Anyway let us assume that we have an explicit solution of form (again we write $\sum_{l=1}^d\psi_{0}(t,u)_lx_l$ for simplicity of notation)
\begin{equation}\label{expform}
\exp\left( \phi_{0}(t,u)+\psi_{0}(t,u)x\right)
\end{equation}  
of the Cauchy problem
\begin{equation}\label{FBK1}
 \left\lbrace \begin{array}{ll}
 \frac{\partial\hat{p}}{\partial t}(t,x,u)= A_0\hat{p}\, (t,x,u),\\
\\
\hat{p}(0,x,u)=\exp(iux),\quad t\ge0,\quad x\in \Omega\subset\mathbb{R}^{d},
\end{array}\right.
\end{equation}
where $\Omega$ is some domain and $A_0$ is some operator 
\begin{align*}
&A_0[f](x)\equiv \frac{1}{2}\sum_{ij}a^0_{ij}(x)\frac{\partial^2}{\partial x_i\partial x_j}f(x)+\sum_i b^0_i(x)\frac{\partial}{\partial x_i}f(x) \\
&
 +\int_{\mathbb{R}^{d}\setminus\lbrace 0\rbrace}\left[  f(x+z)-f(x)-\frac{\partial
f}{\partial x}(x)\cdot z 1_{D}(z)\right]  \nu^0(x,dz).
\end{align*}
Keep in mind that most examples will be of the form (\ref{DiffCauchy}) or maybe (\ref{DiffCauchy}) plus an integral term with Levy measure as in \cite{DFS}. There is no general restriction but the idea is that a certain generalized symbol function is bounded on its domain of definition. More precisely, recall that the ordinary symbol is defined by  
\begin{align*}
&\sigma:\Omega\times i{\mathbb R}^d\rightarrow {\mathbb C},\\
\\
&\sigma(x,iu)= -\frac{1}{2}\sum_{jk}a_{jk}(x)u_ju_k+\sum_j b_j(x)i u_j\\
\\
&\hspace{2cm}+\int_{\mathbb{R}^{d}\setminus \lbrace 0\rbrace}\left[  \exp(iuz)-1-iu\cdot z 1_D(z)\right]  \nu(x,dz)
\end{align*}
for some $\Omega\subseteq {\mathbb R}^d$. Furthermore, the symbol function associated with the operator $A_0$ is
\begin{align*}
&\sigma_0:\Omega\times i{\mathbb R}^d\rightarrow {\mathbb C},\\
\\
&\sigma_0(x,iu)= -\frac{1}{2}\sum_{jk}a^0_{jk}(x)u_ju_k+\sum_j b^0_j(x)i u_j\\
\\
&\hspace{2cm}+\int_{\mathbb{R}^{d}\setminus \lbrace 0\rbrace}\left[  \exp(iuz)-1-iu\cdot z 1_D(z)\right]  \nu^0(x,dz)
\end{align*}
Now asuume that the Cauchy problem with operator $A_0$ has an explicit solution in the form (\ref{expform}).
First we define the generalized symbol function  by
\begin{align*}
&\sigma_g:\Omega\times i{\mathbb R}^d\rightarrow {\mathbb C},\\
\\
&\sigma_g(x,iu):=\sigma(x,\psi_0(t,u)).
\end{align*}
Here the partial derivative $\partial_1$ is the partial derivative with respect to the first variable, i.e. the time variable. Furthermore and similar as in the case of ordinary symbols we define
\begin{align*}
&\sigma_{gl}:i{\mathbb R}^d\rightarrow {\mathbb C},\\
\\
&\sigma_{gl}(iu):=\sigma_l(\psi_0(t,u)),
\end{align*}
where $\sigma_l$ is defined as before. Now there is another interesting effect: using the fact that
\begin{equation}
\partial_t\exp\left(\phi_0(t)+\psi_0(t,iu)x\right)=A_0\exp\left(\phi_0(t)+\psi_0(t,iu)x\right) 
\end{equation}
consider the ansatz $\widehat{p}(t,x,u)=\exp\left(\phi_0(t)+\psi_0(t,iu)x\right)(1+d_k(x,\psi_0(t,u))t^k)$ and plug it into the Cauchy problem
\begin{equation}
\partial_t\widehat{p}=A\widehat{p}
\end{equation}
Similar computations as before (using the fact that $d_k$ has only powers of the spatial variable up to order $k$) leads to the recursion
\begin{equation}
\begin{array}{ll}
d_0=1\\
\\
d_{k+1}=\frac{1}{k+1}\sum_{1\leq |\epsilon | \leq k}
\left(\sigma (x,\psi_0)- \sigma_0(x,\psi_0)\right)  d_k  \\
\\
\hspace{1.3cm}+\frac{1}{k+1}\sum_{0\leq |\epsilon| \leq k}\partial^{\epsilon}_{\xi}\sigma(x,\psi_0) \frac{1}{\epsilon!}\partial^{\epsilon}_xd_k 
\end{array}
\end{equation}
Note that if $\sigma_0=\sigma$, then $d_k=0$ for $k\geq 1$ as espected. Now the point is that we have indeed a bounded symbol, if $\psi$ is bounded. Well it makes sense to rewrite the latter recursion in the following way
\begin{equation}\label{exp1}
\begin{array}{ll}
d_0=1\\
\\
d_{k+1}=\frac{1}{k+1}\sum_{1\leq |\epsilon | \leq k}
\Delta\sigma (x,\psi_0)  d_k + \\
\\
\hspace{1.3cm}\frac{1}{k+1}\sum_{0\leq |\epsilon| \leq k}\partial^{\epsilon}_{\xi}\Delta\sigma(x,\psi_0) \frac{1}{\epsilon!}\partial^{\epsilon}_xd_k+\\
\\
\hspace{1.3cm}\frac{1}{k+1}\sum_{1\leq |\epsilon| \leq k}\partial^{\epsilon}_{\xi}\sigma_0(x,\psi_0) \frac{1}{\epsilon!}\partial^{\epsilon}_xd_k, 
\end{array}
\end{equation}
where $\Delta\sigma(x,\psi)=\sigma(x,\psi_0)-\sigma_0(x,\psi_0)$. Note that in the second sum the multiindex starts for multindices $\epsilon$ with $|\epsilon|\geq 1$. For example, we may have an unbounded symbol where the unboundedness is caused by the drift terms. Then the dffirence symbol is bounded and all derivatives of the symbol function are bounded and hence the symbol expansion holds in the whole domain of the symbol variable $\xi$. Note that this point could not be made if we just consider a brute force expansion for generalized symbol functions via the recursion
\begin{equation}\label{exp2}
\begin{array}{ll}
 d_{k+1}=\frac{1}{k+1}&{\bigg(}(-\partial_1\phi_0(t,u)-x\partial_1\psi_0(t,u))d_k\\
\\
&+\sum_{|\epsilon |\leq k}\partial^{\epsilon}_{\xi}\sigma (x,\psi)\frac{1}{\epsilon !}\partial^{\epsilon}_{\xi}d_k {\bigg )}, 
\end{array}
\end{equation}
where $\partial_1$ denotes the partial derivative with respect to the time variable $t$. Results analogous to theorem 2.11 and theorem 2.13 (or thorem 2.5. or theorem 2.9 in the univariate case) can be written down for either recursion (\ref{exp1}) or (\ref{exp2}). 
As an example one may consider the univariate case where the solution of the equation 
\[
\frac{\partial\widehat{p}}{\partial t}=(a_{0}+a_{1}x)\frac{\partial\widehat{p}}{\partial x^{2}}+(b_{0}+b_{1}x)\frac{\partial\widehat
{p}}{\partial x}
\]
with initial data $\widehat{p}(0,x,u)=\exp(iux)$ can be easily found via solving the corresponding Riccati equations. Not in all cases these solutions correspond to an affine process, of course (but we do not need this assumption, if we consider the solutions to be just solutions of Cauchy problems which have an interpretation to be characteristic functions sometimes).
Or one may consider stochastic volatility models as in \cite{He} where the explicit solution of the corresponding Riccati equations is well known, i.e. Cauchy problems determined by the stochastic differential equation 
\begin{equation}
\begin{array}{ll}
 dX_t=(b_{10}+b_{11}v)dt+\sqrt{v(t)}dW^1_t,\\
\\
dv=(b_{20}-b_{21}v)dt+\sigma\sqrt{v(t)}dW^2_t
\end{array}
\end{equation}
where $W^1_t$ and $W^2_t$ are standard Brownian motions with constant correlation $\rho$
Heston determined the characteristic function to be
\begin{equation}
\widehat{p}(t,x,v;u)=\exp\left( \phi(t,u)+\psi_{01} (t,u)v+iux\right),
\end{equation}
where
\begin{equation}
\begin{array}{ll}
\phi(t,u)=b_{00}iut+\frac{b_{20}}{\sigma^2}\left\lbrace (b_{21}-\rho\sigma iu +d)t-2\ln\left( \frac{1-g\exp(dt)}{1-g} \right) \right\rbrace \\
\\
\psi_{0}(t,u)=\frac{b_{21}-\rho\sigma iu +d}{\sigma^2}\left[\frac{1-\exp(dt)}{1-g\exp(dt)} \right],
\end{array}
\end{equation}
and
\begin{equation}
g=\frac{b_{21}-\rho\sigma iu +d}{b_{21}-\rho\sigma iu -d},~~d=\sqrt{(\rho\sigma iu-b_{21})^2-\sigma^2(2b_{11}iu-u^2)}
\end{equation}
In this case the generalized symbol function would be of the form
\begin{equation}
\sigma_g(x,iu)=\sigma(x,\psi_{0}),~\mbox{where}~\psi_{0}=(\psi_{01},iu)
\end{equation}
Note that such an expansion with generalized symbols may be useful also from a numerical/computational point of view - sometimes we may have faster convergence but this has to be checked from case to case.
Furthermore, we note that our formulas lead to some results about global solutions of Cauchy problems on the flat $n$-torus for operators with affine coefficients which are not covered by the general results of H\"ormander. Only some of that results have an probabilistic interpretation.  

\section{Further remarks}

Generalizations of our formulas for the characteristic function to more general Feller processes seems difficult (if not impossible). In can be shown that on the flat $d$-torus divergence phenomena for expansions considered in this paper are quite persistent (cf. \cite{BKS} for the univariate case). A way out may be to consider other holomorhic transforms (therefore a more general context is considered in \cite{BKS}).
A refined analysis of the affine counting trinagles may reveal that for some processe the power series provided in this paper converge also on unbounded spatial domains. 
Implementation and numerical analysis of the formulas  presented in this paper will be considered in a subsequent paper. From the computational point of view this is very interesting because we can avoid the difficult numerical solution of generalized Riccati equations. A. Eisenbl\"atter (formerly ZIB, Berlin) pointed out to me how the integer recursions considered in this paper may be programmed efficiently using rather sophisticated languages like Haskell etc.. Well this will be considered in a subsequent paper.

\end{document}